\documentclass[11pt,twoside, a4paper]{amsart}

\title
{           {\protect\hfill \normalfont \tiny
            \\ \vspace{10pt}}
Hasse principle and weak approximation for multinorm equations
}

\author{Cyril Demarche and Dasheng Wei
}

\date{\today}


\usepackage[arrow,matrix]{xy}
\usepackage{mathptmx}
\usepackage{amscd}
\usepackage{amssymb}
\usepackage{eucal}
\usepackage{amsxtra}
\usepackage{verbatim}
\usepackage{url}
\usepackage{enumerate}
\usepackage[english]{babel}

\DeclareFontEncoding{OT2}{}{} 
\DeclareTextFontCommand{\textcyr}{\fontencoding{OT2}
    \fontfamily{wncyr}\fontseries{m}\fontshape{n}\selectfont}

\usepackage{comment}

\theoremstyle{plain}

\newtheorem{theorem}{Theorem}
\newtheorem{proposition}[theorem]{Proposition}
\newtheorem{lemma}[theorem]{Lemma}
\newtheorem{corollary}[theorem]{Corollary}

\newtheorem{conditional-result}[theorem]{Conditional Result}

\newtheorem{theorem?}{Theorem(?)} [section]
\newtheorem{proposition?}[theorem]{Proposition(?)}
\newtheorem{lemma?}[theorem]{Lemma(?)}
\newtheorem{corollary?}[theorem]{Corollary(?)}

\newtheorem*{theorem*}{Theorem}
\newtheorem*{proposition*}{Proposition}
\newtheorem*{lemma*}{Lemma}
\newtheorem*{corollary*}{Corollary}
\newtheorem*{question*}{Question}
\newtheorem*{conjecture*}{Conjecture}
\newtheorem*{claim*}{Claim}

\newtheorem*{introtheorem*}{Theorem}
\newtheorem*{introproposition*}{Proposition}
\newtheorem*{introlemma*}{Lemma}
\newtheorem*{introcorollary*}{Corollary}

\theoremstyle{definition}

\newtheorem{example}[theorem]{Example}

\newtheorem*{definition*}{Definition}
\newtheorem*{example*}{Example}

\theoremstyle{remark}
\newtheorem{remark}[theorem]{Remark}

\newtheorem*{remark*}{Remark}

\DeclareSymbolFont{rsfs}{U}{rsfs}{m}{n}
\DeclareSymbolFontAlphabet{\mathcal}{rsfs}

\DeclareFontEncoding{OT2}{}{} 
\DeclareTextFontCommand{\textcyr}{\fontencoding{OT2}
    \fontfamily{wncyr}\fontseries{m}\fontshape{n}\selectfont}
\newcommand{\Sh}{\textcyr{Sh}}

\newcommand{\isoto}{\overset{\sim}{\to}}

\newcommand{\ZZ}{{\mathbb{Z}}}
\newcommand{\QQ}{{\mathbb{Q}}}

\newcommand{\Gal}{{\rm Gal}}

\newcommand{\Br}{{\rm Br}}

\newcommand{\Hom}{{\rm Hom}}

\newcommand{\Gm}{{\mathbf{G}_m}}

\newcommand{\Aut}{{\rm Aut}}

\newcommand{\res}{{\rm res}}

\def\G{{\mathbb{G}}}
\newcommand{\R}{{\textup{R}}}

\def\Z{{\ZZ}}
\def\Q{{\QQ}}

\def\D{{\mathbf{D}}}

\begin{document}

\begin{abstract}
In this note, we are interested in local-global principles for multinorm equations $\prod_{i=1}^n N_{L_i /k}(z_i) = a$
where $k$ is a global field, $L_i/k$ are finite separable field extensions and $a \in k^*$.

In particular, we prove a result relating the Hasse principle and weak approximation for this equation to the Hasse principle and weak approximation for some classical norm equation
$N_{F/k}(w) = a$ where $F := \bigcap_{i=1}^n L_i$. It provides a proof of a "weak approximation" analogue of a recent conjecture by Pollio and Rapinchuk about the multinorm principle. We also provide a counterexample to the original conjecture concerning the Hasse principle.
\end{abstract}

\maketitle

\section{Introduction}
Let $k$ be a global field, $\Omega$ be the set of places of $k$ and $n \geq 2$. Let $L_1, \dots, L_n$ be finite separable field extensions of $k$. We fix a separable closure $\overline{k}$ of $k$ that contains all the $L_i$'s. Throughout this text, intersections of fields and composites of fields are taken inside the given separable closure $\overline{k}$.

For any $a \in k^*$, we consider the following equation
$$\prod_{i=1}^n N_{L_i/k}(z_i) = a \, .$$
It defines an affine $k$-variety $X$, which is a principal homogeneous space under the $k$-torus $T$ defined by the following exact sequence of $k$-tori
\begin{equation}
\label{exseqtori}
0 \to T \to \prod_{i=1}^n \R_{L_i/k} \Gm \xrightarrow{\prod_i N_{L_i/k}} \Gm \to 0 \, ,
\end{equation}
where the last map is the product of norm maps.

In this text, we say that a family of (smooth and geometrically integral) $k$-varieties satisfies the Hasse principle (resp. the Hasse principle and weak approximation) when for every variety $Z$ in this family, if $Z(k_v) \neq \emptyset$ for all $v$, then $Z(k) \neq \emptyset$ (resp.  if $Z(k_v) \neq \emptyset$ for all $v$, then $Z(k) \neq \emptyset$ and $Z(k)$ is dense in $ \prod_{v \in \Omega} Z(k_v)$).

It is well-known that for varieties $X$ as above, the obstruction to the Hasse principle is measured by the finite group $\Sh^2(k, \widehat{T})$ and the obstruction to weak approximation by the finite group $\Sh^2_{\omega}(k, \widehat{T}) / \Sh^2(k, \widehat{T})$, via the Brauer-Manin obstruction (see for instance \cite{S}), where $\widehat{T}$ is the module of characters of $T$.

Recall that for any Galois module $M$ over $k$, we have by definition
$$\Sh^i(k, M) := \textup{Ker}\left(H^i(k,M) \to \prod_{v \in \Omega} H^i(k_v, M)\right)$$
and
$$\Sh^i_{\omega}(k, M) := \left\{ \alpha \in H^i(k,M) \textup{ s.t. } \alpha_v = 0 \textup{ in } H^i(k_v, M) \textup{ for almost all places } v \in \Omega \right\} \, .$$

More precisely, assuming that $X$ has points in every completion of $k$, there is an isomorphism of finite groups $\Sh^1(k,T) \isoto \Hom(\Sh^2(k, \widehat{T}), \Q / \Z)$ (global duality for tori) such that the class of $X$ in $\Sh^1(k,T)$ maps to the Brauer-Manin obstruction to the Hasse principle for $X$, so that one says that the Brauer-Manin obstruction to the Hasse principle for X is the only one. Concerning weak approximation, assuming $X(k) \neq \emptyset$, i.e. assuming that $X$ is $k$-isomorphic to $T$, there is a natural exact sequence, due to Voskresenski{\u\i} (see \cite{V}, 6.38):
\begin{equation}
\label{exseqVos}
0 \to \overline{T(k)} \to \prod_{v \in \Omega} T(k_v) \to \Hom(\Sh^2_{\omega}(k, \widehat{T}) / \Sh^2(k, \widehat{T}), \Q / \Z) \to 0 \, ,
\end{equation}
where  $\overline{T(k)}$ denotes the closure of $T(k)$ inside $\prod_{v \in \Omega} T(k_v)$ (for the product topology), and the last map is defined via the Brauer-Manin obstruction (or via Tate-Nakayama local duality for tori): one says that the Brauer-Manin obstruction to weak approximation on $T$ (or on $X$) is the only one.

The Hasse principle for such a variety $X$ was studied by several authors, including H\"urlimann (see \cite{H}, especially Proposition 3.3), Colliot-Th\'el\`ene and Sansuc (unpublished), Platonov and Rapinchuk (see \cite{PlR}, sections 6.3 and 9.3, and in particular Proposition 6.11), Prasad and Rapinchuk (see \cite{PrR}, section 4, especially Proposition 4.2) and Pollio and Rapinchuk (see \cite{PR}, Main Theorem). The local-global principle for $X$ is for instance related to some arithmetic properties of algebraic groups of type $A_n$ (see for instance \cite{PlR} or \cite{PrR}).

The main result of this note (see Theorem \ref{main thm}) compares the defects of Hasse principle and weak approximation for $X$ to the defects of Hasse principle and weak approximation for the $k$-variety $Y$ defined by $N_{F/k}(w) = a$, where $F := \bigcap_{i=1}^n L_i$, under some technical assumptions. More precisely, if $S$ denotes the norm $k$-torus $\R_{F/k}^1 \Gm$, then we prove that under some assumptions, there is a canonical isomorphism
$$\Sh^2_{\omega}(k, \widehat{S}) \xrightarrow{\simeq} \Sh^2_{\omega}(k, \widehat{T}) \, ,$$
which essentially means that both Hasse principle and weak approximation hold on $X$ if and only if they both hold on $Y$. In other words, we can compute the defect of Hasse principle and weak approximation for the multinorm equation related to $(L_1, \dots, L_n)$ via the defect of Hasse principle and weak approximation for the usual norm equation related to the extension $F/k$.

In particular, it solves an analogue for weak approximation of a conjecture by Pollio and Rapinchuk (see \cite{PR}, section 4), which concerns the multinorm Hasse principle and which we recall here:
\begin{conjecture*}[Pollio-Rapinchuk]
Let $L_1$ and $L_2$ be finite Galois extensions of $k$. If every extension $P$ of $k$ contained in $L_1 \cap L_2$ satisfies the norm principle, then the pair $L_1, L_2$ satisfies the multinorm principle (it may be enough to require that only the intersection $L_1 \cap L_2$ satisfies the norm principle).
\end{conjecture*}

In addition to the proof of the weak approximation analogue of this conjecture, we also provide in Proposition \ref{prop counterex} a counterexample to this conjecture.

\section{The linearly disjoint case}

Throughout this text, $\Gamma_k$ denotes the absolute Galois group of the field $k$. For a $k$-torus $T$, we denote by $\widehat{T}$ the $\Gamma_k$-module of characters of $T$.

We start by the following case, where both Hasse principle and weak approximation hold:
\begin{theorem}
\label{thm trivial intersection}
Let $L_1, \dots, L_n/k$ be finite separable field extensions. Write $\{1, \dots, n\} = I \cup J$, with $I \cap J = \emptyset$ and $I, J \neq \emptyset$.
Let $L_I$ (resp. $L_J$) be the composite of the fields $L_i$, $i \in I$ (resp. $i \in J$). Define $E_I$ (resp. $E_J$) to be the Galois closure of the extension $L_I / k$ (resp. $L_J / k$). Define $T$ to be the $k$-torus of equation $\prod_{i=1}^n N_{L_i /k}(z_i) = 1$.

If $L_I \cap E_J = k$, then
$$\Sh^2_{\omega}(k, \widehat{T}) = 0 \, .$$

In particular, under these assumptions, for any $a \in k^*$, the $k$-variety defined by $\prod_{i=1}^n N_{L_i/k}(z_i) = a$ satisfies the Hasse principle. Assuming it has a rational point, it also satisfies weak approximation.
\end{theorem}

\begin{remark}
Note that the assumption implies that $\cap_i L_i = k$. Note also that this result generalizes section 5 of \cite{PR} and Corollary 2.3 of \cite{W1}, by taking into account more than two field extensions. The proof is inspired by those two results.
\end{remark}

\begin{proof}
Define $M$ to be the $k$-torus $M := M_I \times M_J$ where $M_I := \prod_{i \in I} \R_{L_i/k} \Gm$ (same for $M_J$).
\begin{lemma} \label{lemma 1}
\begin{enumerate}[(i)]
	\item As a $\Gamma_{E_J}$-module (resp. as a $\Gamma_{L_I}$-module), $\widehat{T}$ is a permutation module.
	\item As a $\Gamma_{L_I}$-module, $\widehat{T} \cong \widehat{T}^{\Gamma_{E_J}} \oplus N$, where $N$ is a permutation $\Gamma_{L_I}$-module.
\end{enumerate}
\end{lemma}

\begin{proof}
\begin{enumerate}[(i)]
	\item We have a natural exact sequence of $\Gamma_k$-modules (see the dual exact sequence \eqref{exseqtori}):
\begin{equation}
\label{exseq M T}
0 \to \Z \to \widehat{M} \to \widehat{T} \to 0 \, .
\end{equation}
As a $\Gamma_{E_J}$-module, $\widehat{M_J} \cong \Z^m$ is a trivial module for some integer $m$, therefore $\widehat{T}$ is isomorphic to $\widehat{M_I} \oplus \Z^{m-1}$ as a $\Gamma_{E_J}$-module, hence it is a permutation $\Gamma_{E_J}$-module.

For any $i \in I$, we have an isomorphism of $\Gamma_{L_I}$-modules $\Z[L_i/k] \cong \Z \oplus N_i$, where $N_i$ is a permutation module. Hence by \eqref{exseq M T}, $\widehat{T} \cong \Z^{\# I - 1} \oplus \widetilde{M}$ as $\Gamma_{L_I}$-modules, where $\widetilde{M} := \bigoplus_{i \in I} N_i \oplus M_J$ is a permutation $\Gamma_{L_I}$-module. Therefore $\widehat{T}$ itself is a permutation $\Gamma_{L_I}$-module.
	
	\item Consider the following commutative exact diagram of $\Gamma_k$-modules:
\begin{displaymath}
\xymatrix{
& & 0 \ar[d] & 0 \ar[d] & \\
0 \ar[r] & \Z \ar[r] \ar[d]^= & \widehat{M}^{\Gamma_{E_J}} \ar[r] \ar[d] & \widehat{T}^{\Gamma_{E_J}} \ar[r] \ar[d] & H^1(\Gamma_{E_J}, \Z) = 0 \\
0 \ar[r] & \Z \ar[r] & \widehat{M} \ar[r] \ar[d] & \widehat{T} \ar[r] \ar[d] & 0 \\
& & \widehat{M} / \widehat{M}^{\Gamma_{E_J}} \ar[r]^{\cong} \ar[d] & \widehat{T} / \widehat{T}^{\Gamma_{E_J}} \ar[d] & \\
& & 0 & 0 & \, .
}
\end{displaymath}
Since $L_I \cap E_J = k$, we have $\Z[L_i/k]^{\Gamma_{E_J}} = \Z[L_i/k]^{\Gamma_k} = \Z.\epsilon_i$ for any $i\in I$, where $\epsilon_i := \sum_{g \in \Gamma_k/\Gamma_{L_i}} g$.
We already know that $\Z[L_i/k] \cong \Z \oplus N_i$ as $\Gamma_{L_I}$-modules, therefore, since $\widehat{M_I}/\widehat{M_I}^{\Gamma_{E_J}} \cong \bigoplus_{i \in I} \Z[L_i/k]/{\Z.\epsilon_i}$, we deduce that the map of $\Gamma_{L_I}$-modules $\widehat{M_I} \to \widehat{M_I}/\widehat{M_I}^{\Gamma_{E_J}}$ splits, hence the map of $\Gamma_{L_I}$-modules $\widehat{M} \to \widehat{M}/\widehat{M}^{\Gamma_{E_J}} = \widehat{M_I}/\widehat{M_I}^{\Gamma_{E_J}}$ splits. Therefore the map of $\Gamma_{L_I}$-modules $\widehat{T} \to \widehat{T}/\widehat{T}^{\Gamma_{E_J}}$ also splits, which concludes the proof of Lemma \ref{lemma 1}.
\end{enumerate}
\end{proof}

\begin{lemma} \label{lemma 2}
The restriction map $\rho : H^2(k, \widehat{T}) \to H^2(L_I, \widehat{T}) \oplus H^2(E_J, \widehat{T})$ is injective.
\end{lemma}

\begin{proof}
By point (i) of Lemma \ref{lemma 1}, $H^1(E_J, \widehat{T}) = 0$. Hence the inflation-restriction exact sequence is the following one:
$$0 \to H^2(E_J/k, \widehat{T}^{\Gamma_{E_J}}) \xrightarrow{\textup{inf}_{E_J/k}} H^2(k, \widehat{T}) \xrightarrow{\textup{res}_{k/E_J}} H^2(E_J, \widehat{T}) \, .$$
So it is enough to prove that the composite map
$$\rho' : H^2(E_J/k, \widehat{T}^{\Gamma_{E_J}}) \xrightarrow{\textup{inf}_{E_J/k}} H^2(k, \widehat{T}) \xrightarrow{\textup{res}_{k/L_I}} H^2(L_I, \widehat{T})$$
is injective.
But we have an isomorphism $H^2(E_J / k, \widehat{T}^{\Gamma_{E_J}}) \cong H^2(L_I.E_J / L_I, \widehat{T}^{\Gamma_{E_J}})$ since the natural map $\Gal(E_J.L_I/L_I) \xrightarrow{\cong} \Gal(E_J/k)$ is an isomorphism (because $L_I \cap E_J = k$).
So we can identify the map $\rho'$ with the following composite map
$$H^2(L_I.E_J / L_I, \widehat{T}^{\Gamma_{E_J}}) \xrightarrow{\textup{inf}} H^2(L_I, \widehat{T}^{\Gamma_{E_J}}) \xrightarrow{i_*} H^2(L_I, \widehat{T}) \, ,$$
where the first map is the inflation map for the Galois module $\widehat{T}^{\Gamma_{E_J}}$, and the second map is induced by the inclusion $i : \widehat{T}^{\Gamma_{E_J}} \to \widehat{T}$.
We have $H^1(L_I.E_J, \widehat{T}^{E_J}) = 0$ since $\widehat{T}^{E_J}$ is a constant torsion-free $\Gamma_{E_J}$-module, hence the inflation map $H^2(L_I.E_J / L_I, \widehat{T}^{\Gamma_{E_J}}) \xrightarrow{\textup{inf}} H^2(L_I, \widehat{T}^{\Gamma_{E_J}})$ is injective. By point (ii) of Lemma \ref{lemma 1}, we know that $\widehat{T}^{\Gamma_{E_J}}$ is a direct summand of $\widehat{T}$ as a $\Gamma_{L_I}$-module, hence the natural map $H^2(L_I, \widehat{T}^{\Gamma_{E_J}}) \xrightarrow{i_*} H^2(L_I, \widehat{T})$ is also injective, which concludes the proof.
\end{proof}

We now prove Theorem \ref{thm trivial intersection}.

By Lemma \ref{lemma 2}, the natural restriction map
$$\Sh^2_{\omega}(k, \widehat{T}) \to \Sh^2_{\omega}(L_I, \widehat{T})  \oplus \Sh^2_{\omega}(E_J, \widehat{T})$$
is injective.
By point (i) of Lemma \ref{lemma 1}, we know that $\widehat{T}$ is a permutation $\Gamma_{L_I}$-module (resp. $\Gamma_{E_J}$-module). Therefore we have $\Sh^2_{\omega}(L_I, \widehat{T}) = \Sh^2_{\omega}(E_J, \widehat{T}) = 0$, hence $\Sh^2_{\omega}(k, \widehat{T}) = 0$.
\end{proof}

\begin{remark}
\begin{itemize}
\item Let $k$ be a number field and $K/k$ a Galois extension of group $\bf A_4$ (the alternating group on four elements). Let $L_1$ and $L_2$ be two different degree $4$ subfields of $K$, hence $L_1\cap L_2=k$. However $L_1\cong L_2$ as $k$-algebras, hence $\R_{L_1/k} \Gm \cong \R_{L_2/k} \Gm$ as $k$-tori. Therefore exact sequence \eqref{exseqtori} implies that we have isomorphisms of $k$-tori
    $$T \cong \textup{Ker} \left(\R_{L_1/k} \Gm \times \R_{L_1/k} \Gm \rightarrow \Gm \right)\cong \R_{L_1/k} \Gm \times\R_{L_1/k}^1 \Gm \, ,$$
where the last isomorphism is defined by $(z_1, z_2) \mapsto (z_1, z_1.z_2)$.
    We know that $ \Sh^2_{\omega}(k, \widehat{\R_{L_1/k}^1 \Gm})=\Z / 2 \Z$ by \cite{Kun}, and that $\Sh^2_{\omega}(k, \widehat{\R_{L_1/k}\Gm})=0$ since $\widehat{\R_{L_1/k} \Gm}$ is a permutation module. Therefore $\Sh^2_{\omega}(k, \widehat{T}) =\Z/2\Z$ while $L_1\cap L_2=k$, hence the assumption about the Galois closure is necessary for the conclusion of Theorem \ref{thm trivial intersection} to hold.

\item Following Theorem 4.1 in \cite{CTpreprint} (see also \cite{S}, Remark 1.9.4), for all $a, b \in k^*$ such that $k(\sqrt{a}, \sqrt{b})/k$ is a biquadratic extension, if we define $L_1 := k(\sqrt{a})$, $L_2 := k(\sqrt{b})$ and $L_3 := k(\sqrt{a b})$, then we have $\Sh^2_{\omega}(k, \widehat{T}) = \Z / 2 \Z$ while $L_i\cap L_j=k$ for $1\leq i\neq j \leq 3$. Therefore the assumption $L_I \cap E_J = L_I \cap L_J = k$ is necessary for the conclusion of Theorem \ref{thm trivial intersection} to hold.
\end{itemize}
\end{remark}

\section{The general case}

We now state the main result that deals with a more general situation when the field $\bigcap_{i=1}^n L_i$ is bigger than $k$. In this case, the Hasse principle or weak approximation does not hold in general, but we have the following theorem:

\begin{theorem}
\label{main thm}
Let $L_1, \dots, L_n/k$ be finite separable field extensions.
 Define $F := \bigcap_{i=1}^n L_i$ and assume that $F/k$ is Galois. Define $T$ to be the $k$-torus of equation $\prod_{i=1}^n N_{L_i /k}(z_i) = 1$ and $S$ to be the $k$-torus of equation $N_{F/k}(w) = 1$.
Write $\{1, \dots, n\} = I \cup J$, with $I \cap J = \emptyset$ and $I, J \neq \emptyset$. Let $F_i$  be a field extension of $L_i$ such that the natural map $\Aut_k(F_i)\rightarrow \Aut_k(F)$ is surjective. Let $F_I$ (resp. $F_J$) be the composite of the fields $F_i$, $i \in I$ (resp. $i \in J$). Let $E_I$ (resp. $E_J$) be the Galois closure of the extension $F_I / F$ (resp. $F_J / F$).

If $F_I \cap E_J = F$, then
$$\Sh^2_{\omega}(k, \widehat{S}) \xrightarrow{\simeq} \Sh^2_{\omega}(k, \widehat{T}) \, .$$
\end{theorem}

This theorem implies the following corollary, which proves the "weak approximation analogue" of the conjecture by Pollio and Rapinchuk about Hasse principle for multinorm tori (see the introduction and the conjecture in section 4 of \cite{PR}):
\begin{corollary}
\label{cor1}
Under the same assumptions as in Theorem \ref{main thm}, let $a \in k^*$. Assume that the $k$-variety of equation $\prod_i N_{L_i/k}(z_i) = a$ has a $k$-point. Then weak approximation holds for the equation $\prod_i N_{L_i/k}(z_i) = a$ if it holds for the equation $N_{F/k}(w) = a$.
\end{corollary}

\begin{proof}
Assume that weak approximation holds for the equation $N_{F/k}(w) = a$. It is equivalent to say that $\Sh^2_{\omega}(k, \widehat{S}) / \Sh^2(k, \widehat{S}) = 0$. Theorem \ref{main thm} implies that the natural map
$$\Sh^2_{\omega}(k, \widehat{S}) / \Sh^2(k, \widehat{S}) \to \Sh^2_{\omega}(k, \widehat{T}) / \Sh^2(k, \widehat{T})$$
is surjective, hence $\Sh^2_{\omega}(k, \widehat{T}) / \Sh^2(k, \widehat{T}) = 0$, which implies by Voskresenski{\u\i}'s exact sequence \eqref{exseqVos} that $T$, hence the $k$-variety of equation $\prod_i N_{L_i/k}(z_i) = a$, satisfies weak approximation.
\end{proof}

We also get a particular case of their conjecture concerning the multinorm Hasse principle:
\begin{corollary}
\label{cor2}
Under the same assumptions as in Theorem \ref{main thm}, assume that the Hasse principle and weak approximation hold for equations $N_{F/k}(w) = a$ (for all $a \in k^*$). Then the Hasse principle and weak approximation hold for equations $\prod_i N_{L_i/k}(z_i) = a$ (for all $a \in k^*$).
\end{corollary}
In particular, this corollary contains the case of two Galois extensions $L_1, L_2$ of $k$ such that $L_1 \cap L_2$ is a cyclic field extension of $k$: this case was one motivation for the conjecture in \cite{PR} (see the remark after the conjecture in section 4 of \cite{PR}).

\begin{proof}
The assumption means exactly that $\Sh^2_{\omega}(k, \widehat{S}) = 0$. Hence Theorem \ref{main thm} implies that $\Sh^2_{\omega}(k, \widehat{T}) = 0$, which means that both the Hasse principle and weak approximation hold for equations $\prod_i N_{L_i/k}(z_i) = a$ (for all $a \in k^*$).
\end{proof}

We now prove Theorem \ref{main thm}.

\begin{proof}

Let $R'$ be the $F$-torus defined by the equation $\prod_i N_{L_i /F}(z_i) = 1$ and $R:= \R_{F/k}(R')$.

We have an exact sequence of $k$-tori:
$$0 \to R \to T \to S \to 0 \, ,$$
where the morphism $T \to S$ is given by $w = \prod_i N_{L_i/F}(z_i)$.

The dual exact sequence of Galois modules
$$0 \to \widehat{S} \to \widehat{T} \to \widehat{R} \to 0$$
induces a long exact sequence
$$H^1(k,\widehat{R}) \to H^2(k,\widehat{S}) \to H^2(k,\widehat{T}) \to H^2(k,\widehat{R}) \, .$$
We know that $H^i(k,\widehat{R}) = H^i(F,\widehat{R'})$.

We first prove that $\Sh^2_{\omega}(k, \widehat{R}) = 0$. First, we have a canonical isomorphism $\Sh^2_{\omega}(k, \widehat{R}) \cong \Sh^2_{\omega}(F, \widehat{R'})$. Since the Galois closure of $L_J / F$ is contained in the Galois extension $E_J / F$, the assumption $F_I \cap E_J = F$ implies that the intersection between $L_I$ and the Galois closure of $L_J/F$ is also $F$, therefore the field extensions $L_i/F$ fulfill the assumptions of Theorem  \ref{thm trivial intersection} (over the base field $F$). Therefore Theorem  \ref{thm trivial intersection} ensures that $\Sh^2_{\omega}(F, \widehat{R'}) = 0$ (see also \cite{PR}, section 5 or \cite{W1}, Corollary 2.3 in some particular cases).

Therefore $\Sh^2_{\omega}(k, \widehat{T})$ is contained in the image of the map $H^2(k,\widehat{S}) \to H^2(k,\widehat{T})$.

Let us prove now that the map $H^2(k,\widehat{S}) \to H^2(k,\widehat{T})$ is injective. We have
$$H^1(k,\widehat{R}) = H^1(F,\widehat{R'}) = \textup{Ker}(H^2(F, \Z) \to \bigoplus_i H^2(L_i, \Z)) \, .$$
Since $F =\bigcap_i  L_i$, we know that the map $H^2(F, \Z) \to \bigoplus_i H^2(L_i, \Z)$ is injective, therefore $H^1(k,\widehat{R}) = 0$, hence the map $H^2(k,\widehat{S}) \to H^2(k,\widehat{T})$ is injective.

Now we start to prove that the kernel of the map $\psi : H^2(k, \widehat{S}) \to H^2(k, \widehat{T}) / \Sh^2_{\omega}(k, \widehat{T})$ is equal to $\Sh^2_{\omega}(k, \widehat{S})$. This kernel clearly contains $\Sh^2_{\omega}(k, \widehat{S})$. Let us prove the converse inclusion.

First we show that $\Sh^2_{\omega}(F, \widehat{T_F})=0$.

Since $F/k$ is Galois, $L_i \otimes_k F\cong \prod_{\sigma \in \Gal(F/k)} L_i^{\sigma} \cong \prod_{\sigma \in \Gal(F/k)} \tilde{\sigma}(L_i)$ as $F$-algebras, where $L_i^{\sigma}$ denotes the field $L_i$ endowed with the $F$-algebra structure given by the morphism $F \xrightarrow{\sigma} F \to L_i$, and $\tilde{\sigma} \in \Aut_k(F_i)$ is a (chosen) lift of $\sigma$ (by assumption, the natural map $\Aut_k(F_i)\rightarrow \Gal(F/k)$ is surjective).

Therefore, the $F$-torus $T_F$ is defined by the following equation inside $\prod_{i, \sigma} R_{\tilde{\sigma}(L_i)/F} \Gm $:
$$\prod_{i, \sigma} N_{\tilde{\sigma}(L_i)/F}(z_{i, \sigma}) = 1 \, .$$

We first notice that $\bigcap_{i, \sigma} \tilde{\sigma}(L_i) = F$. Let $\tilde L_I$ (resp. $\tilde L_J$) be the composite of the fields $\tilde{\sigma}(L_i)$ where $\sigma\in \Gal(F/k)$ and $i\in I$ (resp. $i\in J$). Since $\tilde{\sigma}(L_i) \subset F_i$, it implies that $\tilde L_I\subset  F_I$ and $\tilde L_J\subset F_J$. Since  $F_I\cap E_J=F$, we have $\tilde L_I \cap E_J = F$, hence the intersection between $\tilde L_I$ and the Galois closure of $\tilde L_J / F$ is also $F$, therefore $\Sh^2_{\omega}(F, \widehat{T_F})=0$ by Theorem \ref{thm trivial intersection} applied to the field extensions $\tilde{\sigma}(L_i)/F$.

Consider the following commutative diagram
\begin{displaymath}
\xymatrix{
H^2(k, \widehat{S}) \ar[r]^{\psi} \ar[d] & H^2(k, \widehat{T}) / \Sh^2_{\omega}(k, \widehat{T}) \ar[d] \\
H^2(F, \widehat{S_F}) \ar[r]^(.3){\psi_F} & H^2(F, \widehat{T_F}) / \Sh^2_{\omega}(F, \widehat{T_F}) = H^2(F, \widehat{T_F}) \, .
}
\end{displaymath}

We now prove that the map $\psi_F$ in the above diagram is injective, by showing that $H^1(F, \widehat{R}) = 0$.
Note that by definition $R_F$ is identified with the kernel of the product of norm maps
$$\prod_{i=1}^n \R_{L_i \otimes_k F / F} \G_m \to \R_{F \otimes_k F / F} \G_m \, ,$$
hence we get an isomorphism
$$H^1(F,\widehat{R}) = \bigoplus_{\sigma \in \Gal(F/k)} \textup{Ker} \left(H^2(F, \Z) \to \bigoplus_{i} H^2(\tilde{\sigma}(L_i), \Z)\right) \, .$$
Since  $F_I\cap  E_J=F$, it implies that $F = \bigcap_{i}  \tilde{\sigma}(L_i)$ for all $\sigma \in \Gal(F/k)$. Therefore the map $H^2(F, \Z) \to \bigoplus_{i} H^2(\tilde{\sigma}(L_i), \Z)$ is injective  for all $\sigma \in \Gal(F/k)$, hence $H^1(F,\widehat{R}) = 0$.

Let $\alpha \in \textup{Ker}(\psi)$. Then $\alpha_F = 0$ in $H^2(F, \widehat{S_F})$. Hence $\alpha \in H^2(F/k, \widehat{S})$.
Considering the long exact sequence associated to the exact sequence
$$0 \to \Z \to \Z[F/k] \to \widehat{S} \to 0 \, ,$$
we get an isomorphism $H^2(F/k, \widehat{S}) \xrightarrow{\simeq} H^3(F/k, \Z)$. For any $g \in \textup{Gal}(F/k)$, we have $H^3(<g>, \Z) = 0$, hence we get that $H^2(F/k, \widehat{S}) \subset \Sh^2_{\omega}(k, \widehat{S})$ (in fact, this is an equality). Hence $\alpha \in \Sh^2_{\omega}(k, \widehat{S})$, i.e. $\textup{Ker}(\psi) \subset \Sh^2_{\omega}(k, \widehat{S})$.

It concludes the proof: the map $H^2(k, \widehat{S}) \to H^2(k, \widehat{T})$ induces an isomorphism
$$\Sh^2_{\omega}(k, \widehat{S}) \cong \Sh^2_{\omega}(k, \widehat{T}) \, .$$
\end{proof}

\begin{example}
\begin{enumerate}[(i)]
	\item The assumptions of Theorem \ref{main thm} hold if $n=2$ and $L_1, L_2/k$ are Galois.
	\item They also hold if $n \geq 2$, $L_1, \dots , L_n/k$ are Galois and $(L_1 \dots L_r) \cap (L_{r+1} \dots L_{n}) = F$ (for some $1 \leq r < n$).
	\item Let $L_I$ (resp. $L_J$) be the composite of the fields $L_i$, $i \in I$ (resp. $i \in J$). Let $\tilde{E}_I$ (resp. $\tilde{E}_J$) be the Galois closure of the extension $L_I / k$ (resp. $L_J / k$). The assumptions of Theorem \ref{main thm} also hold if $\tilde{E}_I \cap \tilde{E}_J = F$.

    \item Let $k=\Q$ and $L_1 := \Q(\sqrt{2},\sqrt[4]{3})$ and $L_2 := \Q(\sqrt[4]{2})$. Then $F = L_1 \cap L_2 =\Q(\sqrt{2})$. We can choose $F_1=L_1$ and $F_2 = \Q(\sqrt{-1}, \sqrt[4]{2})$. Since $E_2=F_2$, it implies $F_1 \cap E_2= F$. Therefore we get $\Sh^2_{\omega}(\widehat{T}) \cong \Sh^2_{\omega}(k, \widehat{S}) = 0$ (since $F/\Q$ is cyclic) while $\tilde{E}_1\cap \tilde{E}_2=\Q(\sqrt{-1},\sqrt{2})\neq F$ since $\tilde{E}_1=\Q(\sqrt{-1},\sqrt{2},\sqrt[4]{3})$ and $\tilde{E}_2=E_2$. Therefore Theorem \ref{main thm} is more general than point (iii).

\end{enumerate}
\end{example}

The following example shows that even in the case of two field extensions, some assumptions about Galois closures have to be made for Theorem \ref{main thm} to hold.
\begin{example}
Let $k$ be a number field, $a\in k \text{ and }b, d\in k^*$. Let $L_1=k(\sqrt{a-b\sqrt{d}})$ be  a field extension of $k$ of degree $4$. Suppose $m:=a^2-b^2d$ is not a square in $k(\sqrt{d})$ ($eg.$ $k=\Q \text { and } (a,b,d,m)=(1,1,2,-1)$), hence $L_1/k$ is non-Galois. Let $L_2=k(\sqrt{d},\sqrt{m})$. Then $F=L_1\cap L_2=k(\sqrt{d})$ and
$$\Sh^2_{\omega}(k, \widehat{T})=\Z/2\Z \text{ while } \Sh^2_{\omega}(k, \widehat{S})=0.$$
\end{example}

\begin{proof} Note that $F_1\supset L_1(\sqrt{m})$, hence $F_1\cap F_2\supset F_1\cap L_2=L_2\neq F$, hence the assumptions of Theorem \ref{main thm} are not satisfied.

Let $G := \Gal(L/k) \cong \D_4$ where $L := L_1.L_2$ is the Galois closure of $L_1/k$, let $\widehat{M}$ be the permutation $G$-module $\Z[L_2/k]$ and $\widehat{N} := \Z[L_1/k] / \Z$. Then we have an exact sequence of $G$-modules
$$0 \to \widehat{M} \to \widehat{T} \to \widehat{N} \to 0 \, ,$$
hence an exact sequence of cohomology groups:
$$0 = H^1(G, \widehat{M}) \to H^1(G, \widehat{T}) \xrightarrow{\rho} H^1(G, \widehat{N}) \to H^2(G, \widehat{M}) \to H^2(G, \widehat{T}) \to H^2(G, \widehat{N}) \, .$$
But we have isomorphisms of finite groups
$$H^1(G, \widehat{T}) \cong \textup{Ker}(H^1(L/k, \Q / \Z) \to H^1(L/L_1, \Q/\Z) \oplus H^1(L/L_2, \Q/\Z)) = H^1(F/k, \Q / \Z)$$
and
$$H^1(G, \widehat{N}) \cong \textup{Ker}(H^1(L/k, \Q / \Z) \to H^1(L/L_1, \Q/\Z)) = H^1(F/k, \Q / \Z) \, ,$$
therefore $\rho$ is an isomorphism. Hence we get a commutative diagram with exact rows
\begin{equation} \label{diag sha ex}
\xymatrix{
0 \ar[r] & H^2(G, \widehat{M}) \ar[r] \ar[d]^{\phi = (\phi_g)} & H^2(G, \widehat{T}) \ar[r] \ar[d] & H^2(G, \widehat{N}) \ar[d] \\
& \prod_{g \in G} H^2(\langle g \rangle, \widehat{M}) \ar[r]^{\psi = (\psi_g)} & \prod_{g \in G} H^2(\langle g \rangle, \widehat{T}) \ar[r] & \prod_{g \in G} H^2(\langle g \rangle, \widehat{N}) \, .
}
\end{equation}
Since $G \cong \D_4$ and $L$ is the Galois closure of $L_1/k$, Proposition 1 in \cite{Kun} implies that $\Sh^2_{\omega}(k, \widehat{N}) = 0$. So we deduce from an easy diagram chase in \eqref{diag sha ex} that
$$\textup{Ker}(\psi \circ \phi) \xrightarrow{\simeq} \Sh^2_{\omega}(k, \widehat{T}) \, .$$

We now prove that $\psi \circ \phi= 0$.
For any $g \in G$, let $L^g$ denote the subfield of $L$ fixed by $g$. We consider the two following cases:
\begin{enumerate}[(i)]
	\item if $g \notin \Gal(L/F')$, where $F' := k(\sqrt{m})$. Then $g$ has order $2$ (since $G \cong \D_4$), and $L^g \neq L_2$. Then there exists a quadratic extension $K'/k$ contained in $L_2$ such that $L^g.K' = L$. Therefore $L^g \otimes_k L_2 = L \otimes_{K'} L_2 = L \oplus L$. Therefore
$$H^2(\langle g \rangle, \widehat{M}) = H^2(L/L^g, \Z[L_2/k]) = H^2(L/L^g, \Z^2[L/L^g]) = 0 \, ,$$
hence $\phi_g = 0$ and in particular $\psi_g \circ \phi_g = 0$.
	\item if $g \in \Gal(L/F') \cong \Z / 4 \Z$, then by functoriality, it is enough to consider the case when $\langle g \rangle = \Gal(L/F')$ (if $\psi_g \circ \phi_g = 0$ for some $g$ such that $\langle g \rangle = \Gal(L/F')$, then $\psi_{g'} \circ \phi_{g'} = 0$ for all $g' \in \Gal(L/F')$, since $\psi_{g'} \circ \phi_{g'}$ factors through $\psi_g \circ \phi_g$). So we now assume that $g$ has order $4$. Let $\widehat{M'} := \Z[L_2/k] / \Z$ and $\widehat{N'} := \Z[L_1/k]$. We have a natural exact sequence:
$$
0 \to \widehat{N'} \to \widehat{T} \to \widehat{M'} \to 0 \, ,
$$
hence an exact sequence
$$H^2(L/F', \widehat{N'}) \to H^2(L/F', \widehat{T}) \to H^2(L/F', \widehat{M'}) \, .$$
But we have $H^2(L/F', \widehat{N'}) = H^2(L/F', \Z[L_1/k]) = 0$ since $F'.L_1 = L$. Hence $H^2(L/F', \widehat{T}) \to H^2(L/F', \widehat{M'})$ is injective. Therefore it is enough to prove that the composite map
$$H^2(G, \Z[L_2 / k]) \xrightarrow{\phi_g} H^2(L/F', \Z[L_2/k]) \xrightarrow{\psi'_g} H^2(L/F', \widehat{M'})$$
is the zero map. But $\psi'_g$ factors through the cokernel of the natural map $i_g : H^2(L/F', \Z) \to H^2(L/F', \Z[L_2/k])$, hence we only need to prove that the composite map
$$H^2(G, \Z[L_2 / k]) \xrightarrow{\phi_g} H^2(L/F', \Z[L_2/k]) \xrightarrow{\psi''_g} \textup{Coker}(i_g)$$
is zero.

The image of $\phi_g$ in $H^2(L/F', \Z[L_2/k])$ is $\Gal(F'/k)$-invariant, hence we only need to show that $\psi''_g$ restricted to $H^2(L/F', \Z[L_2/k])^{\Gal(F'/k)}$ is zero. Since $\Z[L_2/k] \cong \Z[F/k]\otimes\Z[F'/k]$ canonically as $G$-modules, it implies that $$H^2(L/F', \Z[L_2/k])\cong H^2(L/F', \Z[F/k])\otimes\Z[F'/k]$$
as $\Gal(F'/k)$-modules. Let $\sigma$ is the unique nontrivial element of $\Gal(F'/k)$, then  $\sigma$ induces an isomorphism $H^2(L/F', \Z[F/k])\rightarrow H^2(L/F', \Z[F/k]),\chi \mapsto \chi^{\sigma}$. Let $$(\chi_1,\chi_2)\in  \prod_{\Gal(F'/k)}H^2(L/F', \Z[F/k])\cong H^2(L/F', \Z[F/k])\otimes\Z[F'/k].$$ Then $\sigma(\chi_1,\chi_2)=(\chi_2^{\sigma},\chi_1^{\sigma})$ by the definition of the action of $\Gal(F'/k)$ on $H^2(L/F', \Z[L_2/k])$. Since $H^2(L/F', \Z[F/k])\cong H^2(L/L_2,\Z)\cong \Z/2\Z$, we have $\chi_i^{\sigma}=\chi_i$ for $i=1,2$. Therefore $$H^2(L/F', \Z[L_2/k])^{\Gal(F'/k)}=\{(\chi,\chi)\mid \chi\in H^2(L/F', \Z[F/k])\},$$ hence it is just the image of $i_g$, so its image by $\psi''_g$ is zero.

Then we deduce that $\psi''_g \circ \phi_g = 0$, hence $\psi_g \circ \phi_g = 0$.
\end{enumerate}
So we proved that $\psi \circ \phi = 0$, hence we get
$$\Sh^2_{\omega}(k, \widehat{T}) = H^2(G, \widehat{M}) \cong H^2(L/L_2, \Z) \cong \Z / 2\Z,$$
which concludes the proof.
\end{proof}

\begin{remark}
We give here a heuristic reason why corollaries \ref{cor1} and \ref{cor2} hold, while the original conjecture by Pollio and Rapinchuk does not (see Proposition \ref{prop counterex} below).

From a geometric point of view, if we denote by $X$ the $k$-variety defined by the equation $\prod_{i=1}^n N_{L_i/k}(z_i) = a$ and by $Y$ the $k$-variety defined by $N_{F/k}(w)=a$, then we have a natural morphism $\pi : X \to Y$, defined by $w = \pi((z_i)) := \prod_{i=1}^n N_{L_i/F}(z_i)$. Clearly, the map $\pi : X \to Y$ is a torsor under the $k$-torus $R$ defined at the beginning of the proof of Theorem \ref{main thm}. Theorem \ref{thm trivial intersection} ensures that the rational fibers of the morphism $\pi$ (which are $k$-torsors under $R$) satisfy the Hasse principle and weak approximation. In the statements of the conjecture or of the corollaries, one deduces local-global principles for the total space $X$ of the fibration $\pi$ from local-global principles for the base space $Y$ of this fibration. A classical way to prove such results (since we know that the fibers of $\pi$ do satisfy local-global principles) is the fibration method (see for instance \cite{CT}, section 3).

And the key point is that in general, for such a fibration $\pi : X \to Y$, one cannot deduce the Hasse principle for $X$ from the Hasse principle for the basis $Y$ and the Hasse principle for rational fibers of $\pi$.  But by classical results, one can sometimes prove that the Hasse principle holds for $X$ \emph{assuming that $Y$ satisfies both the Hasse principle and weak approximation} (see \cite{CT}, section 3). And indeed, in the multinorm situation, Proposition \ref{prop counterex} below shows that the fibration method and the Hasse principle on $X$ fail due to the failure of weak approximation on $Y$, while Theorem \ref{main thm} implies that under the stronger assumption that $Y$ does satisfy the Hasse principle and weak approximation, the variety $X$ does satisfy the Hasse principle.

However, we did not manage to write a direct geometric proof of Theorem \ref{main thm}, nor of its corollaries, via the fibration method and Theorem \ref{thm trivial intersection}, since this method requires either that the basis $Y$ of the fibration is proper or that $Y$ satisfies strong approximation, which is not the case in our situation. Nevertheless, the general framework of fibration methods gives an explanation why our results hold while the original conjecture does not.

\end{remark}

\section{A counterexample to the original conjecture}

We now construct a counterexample to the original conjecture of Pollio and Rapinchuk (see introduction), relative to the multinorm Hasse principle:

\begin{proposition} \label{prop counterex}
Let $k=\Q$. Let $q=2$ or $q\equiv 5 \mod 8$ be a prime. Suppose $m$ is an integer and
$$m\equiv \begin{cases}\pm 1\mod 8\text { or }\pm 2 \mod 16, &\text{ if }q=2,\cr \pm 1,\pm 5 \mod 8, &\text{ if }q\equiv 5 \mod 8,\end{cases}$$
and none of $\pm m,\pm mq$ is a square in $\Q^*$ ($eg.$ $(q,m)=(2,7) \text{ or } (5,17)$). Let $L_1 = \Q(\sqrt{-1}, \sqrt[4]{q})$ and $L_2 = \Q(\sqrt{-1}, \sqrt{m \sqrt{q}})$. Then $F =L_1\cap L_2= \Q(\sqrt{-1}, \sqrt{q})$, and $\Sh^2(\Q, \widehat{T}) = \Z / 2 \Z$ while $\Sh^2(\Q, \widehat{S}) = 0$.

In particular, any subextension of $F/\Q$ satisfies the Hasse norm principle, but $(L_1, L_2; \Q)$ does not satisfy the multinorm principle, i.e. there exists $a \in \Q^*$ such that the equation $N_{L_1/\Q}(z_1). N_{L_2/\Q}(z_2) = a$ violates the Hasse principle, while the equation $N_{F/\Q}(w) = a$ has a rational solution.
\end{proposition}

\begin{proof}Both $L_1$ and $L_2$ are Galois over $\Q$. By assumption, $\pm m,\pm mq$ are not squares in $\Q^*$, hence $L_1\neq L_2$.

First, Sansuc proved that $\Sh^2_{\omega}(\Q, \widehat{S}) = \Z / 2 \Z$ (see \cite{S}, (2.16)). Hence by Theorem \ref{main thm}, we get $\Sh^2_{\omega}(\Q, \widehat{T}) = \Z / 2 \Z$. Similarly, we know that in this case $\Sh^2(\Q, \widehat{S}) = 0$ since $\Gal(F/\Q)=\Gal(F_2/\Q_2)$, where $F_2=F\otimes_\Q \Q_2$ is a field.

By Theorem \ref{main thm}, we have the following commutative diagram
\begin{displaymath}
\xymatrix{
H^2(F/\Q, \widehat{S}) = \Sh^2_{\omega}(\Q, \widehat{S}) \ar[r]^(.6){\cong} \ar[d]^f & \Sh^2_{\omega}(\Q, \widehat{T}) \ar[d] \\
\prod_p H^2(F_p/\Q_p, \widehat{S}) \ar[r]^{g=(g_p)} & \prod_p H^2(\Q_p, \widehat{T}) \, ,
}
\end{displaymath}
where $F_p := F \otimes_{\Q} \Q_p$. This implies that $\Sh^2(k, \widehat{T}) \cong \textup{Ker}(g \circ f)$.

If $p \neq 2$, then $F_p$ is a product of cyclic field extensions of $\Q_p$, therefore $H^3(F_p/\Q_p, \Z) = 0$. We know that $H^2(F_p/\Q_p, \widehat{S}) \cong H^3(F_p/\Q_p, \Z)$, therefore for all odd $p$, $H^2(F_p/\Q_p, \widehat{S})=0$.

Hence $\Sh^2(k, \widehat{T}) \cong \textup{Ker}(g_2 \circ f)$. Therefore, we only need to prove that $g_2 = 0$.

Since $m$ is a square in $\Q_2(\sqrt{-1},\sqrt{q})$, then $L_1 \otimes_{\Q} \Q_2 = L_2 \otimes_{\Q} \Q_2$. Denote this degree $8$ field extension of $\Q_2$ by $L_v$. Note that $L_v / \Q_2$ is a degree $8$ Galois extension, with Galois group isomorphic to the dihedral group $\D_4$.

We deduce that $T_2 := T \times_{\Q} \Q_2$ is isomorphic as a $\Q_2$-torus to the torus defined by the equation $N_{L_v/\Q_2}(w_1). N_{L_v/\Q_2}(w_2) = 1$, so $T_2 \cong \R_{L_v/\Q_2} \Gm \times \R_{L_v/\Q_2}^1 \Gm$. Let $T' := \R_{L_v/\Q_2}^1 \Gm$. Then the natural map $T_2 \to S_2$ factors through $T'$, hence we only need to prove that the map $h : H^2(F_2/\Q_2, \widehat{S_2}) \to H^2(L_v/\Q_2, \widehat{T'})$ is zero.

Define $G := \Gal(L_v/\Q_2) \cong \D_4$, $H := \Gal(L_v/F_2) \cong \Z / 2 \Z$, so that $G/H \cong \Gal(F_2/\Q_2) \cong \Z / 2 \Z \times \Z / 2 \Z$. We have an commutative exact diagram of $G$-modules:
\begin{displaymath}
\xymatrix{
0 \ar[r] & \Z \ar[r] \ar[d]^= & \Z[G/H] \ar[r] \ar[d] & \widehat{S_2} \ar[r] \ar[d] & 0 \\
0 \ar[r] & \Z \ar[r] & \Z[G] \ar[r] & \widehat{T'} \ar[r] & 0
}
\end{displaymath}
that induces the following commutative diagram
\begin{displaymath}
\xymatrix{
H^2(G/H, \widehat{S_2}) \ar[r]^{\cong} \ar[d]^{h} & H^3(G/H, \Z) \ar[d]^{\textup{inf}} \\
H^2(G, \widehat{T'}) \ar[r]^{\cong} & H^3(G, \Z) \, .
}
\end{displaymath}
Therefore we only need to prove that the inflation map $H^3(G/H, \Z) \to H^3(G, \Z)$, i.e. the inflation map $H^2(G/H, \Q / \Z) \to H^2(G, \Q / \Z)$, is zero. Consider the restriction-inflation exact sequence
\begin{equation} \label{res-inf}
H^1(G, \Q / \Z) \xrightarrow{\textup{res}} H^1(H, \Q / \Z)^{G/H} \xrightarrow{\delta} H^2(G/H, \Q / \Z) \xrightarrow{\textup{inf}} H^2(G, \Q / \Z) \, .\end{equation}
Since $H$ is exactly the derived subgroup of $G \cong \D_4$, the restriction map $H^1(G, \Q / \Z) \xrightarrow{\textup{res}} H^1(H, \Q / \Z)^{G/H}$ is the zero map. So the map $\delta$ is injective. Moreover, $H \cong\Z / 2 \Z$ and $H$ is central in $G$, therefore $H^1(H, \Q / \Z)^{G/H} \cong \Z / 2 \Z$.
A classical computation of Schur (see for instance \cite{Ka}, Corollary 2.2.12) implies that
$$H^2(G/H, \Q/\Z) = H^2(\Z / 2 \Z \times \Z / 2 \Z, \Q / \Z) \cong \Z / 2 \Z \, ,$$
hence the map $\delta$ is surjective.

Eventually, the exact sequence \eqref{res-inf} implies that the map $\textup{inf}$ is zero, which concludes the proof.
\end{proof}

We keep the same notations as in Proposition \ref{prop counterex}.

We provide here an explicit example for Proposition \ref{prop counterex}, i.e. an explicit rational number $a$ for which the equation $N_{L_1/\Q}(z_1). N_{L_2/\Q}(z_2) = a$ violates the Hasse principle, while the equation $N_{F/\Q}(w) = a$ has a rational solution.

\begin{example}
We consider the special case where $q=2$ and $m=7$, that is $L_1 := \Q(\sqrt{-1}, \sqrt[4]{2})$ and $L_2 := \Q(\sqrt{-1}, \sqrt{7\sqrt{2}})$ and $F = L_1 \cap L_2 = \Q(\sqrt{-1}, \sqrt{2})$. Then the equation $N_{L_1/\Q}(z_1). N_{L_2/\Q}(z_2) = 97$ is a counterexample to the Hasse principle (while $N_{F/\Q}(3+\sqrt{2} + \sqrt{-2}) = 97$).

Let us prove this fact.

We first prove that the equation $N_{L_1/\Q}(z_1). N_{L_2/\Q}(z_2) = 97$ is locally solvable. First, it is clearly solvable over $\Q_p$, for all $p \neq 2, 97$, and over $\mathbb{R}$. Since $97 \equiv 1 \, \, [32]$, the equation $x^8 = 97$ is solvable over $\Q_2$, hence the equation $N_{L_1/\Q}(z_1). N_{L_2/\Q}(z_2) = 97$ has a solution over $\Q_2$. We check that $\left( \frac{-1}{97} \right) = \left( \frac{2}{97} \right) = 1$, and that $\left( \frac{2}{97} \right)_4 = -1$ and $\left( \frac{7}{97} \right) = -1$, therefore the extension $L_2/\Q$ is totally split at the prime $97$, hence the equation $N_{L_1/\Q}(z_1). N_{L_2/\Q}(z_2) = 97$ is solvable over $\Q_{97}$. Eventually, the multinorm equation under consideration is locally solvable.

We now prove that this equation does not have any global solution. Following \cite{W2} Corollary 1, we consider the following element $A := \textup{cor}_{F/\Q}(N_{L_2/F}(z_2),\chi)$ in the Brauer group $\Br(X)$, where $\chi$ is the unique non-trivial character of $\Gamma_F$ that factors through $\Gal(L_1/F)$. We now prove that this element $A$ induces a Brauer-Manin obstruction to the Hasse principle on $X$.

Let $p$ be a prime number (or $p=\infty$) and let $x_p = (z_{1,p}, z_{2,p}) \in X(\Q_p)$ with $z_{i,p} \in (L_i \otimes \Q_p)^*$.
\begin{itemize}
	\item if $p=2$, then $A(x_p) = 0$ since $(N_{L_2/F}(z_{2,2}),\chi) = 0$, because $L_1 \otimes \Q_p = L_2 \otimes \Q_p$ and $(N_{L_1/F}(z),\chi) = 0$.
	\item if $p=\infty$, then $\chi_p$ is trivial since $F$ is totally imaginary, hence $A(x_p)=0$.
\end{itemize}
In the following, we will show:
\begin{itemize}
    \item if $p \neq 2, 97, \infty$, then $A(x_p)=0$.
	\item if $p = 97$, then $A(x_p) \neq 0$.
\end{itemize}

Let $p\neq 2,\infty$ and let us prove that the restriction of $A$ to $\Br(X\times_\Q \Q_p)$ is a constant. We fix an embedding $\bar \Q \hookrightarrow \bar \Q_p$ and let $K=\Q_p\cap F$, more precisely, we will show that the restriction $\res_{\Q/K}(A)$ of $A$ to $\Br(X\times_\Q K)$ is the constant $(97, \chi')$, where $\chi'$ is a character of $\Gamma_K$ which lifts $\chi$.

Using the notation in Proposition 1.5.6 (\cite{NSW}, chaper 1), let $G=\Gamma_\Q,U=\Gamma_K,V=\Gamma_F\subset U$, we have the following decomposition $$G=\dot{\bigcup}_\sigma U\sigma=\dot{\bigcup}_\sigma U\sigma V \, ,$$
where $\sigma$ runs through a finite system of representatives of the double cosets.

  Then we have
$$\begin{aligned}
\res_{\Q/K}(A)&=\res_{G/U}\text{cor}_{V/G}(N_{L_2/F}(z_2),\chi)
=\sum_\sigma \text{cor}_{U/V}(N_{L_2/F}(z_2)^{\sigma},\chi^\sigma)\\
&=\sum_\sigma \text{cor}_{U/V}(N_{L_2/F}(z_2)^{\sigma},\chi) \, ,
\end{aligned}$$
where the last equation holds since $\chi=\chi^\sigma$ (note that $\chi$ has order $2$).
The field extension $F/K$ is cyclic since $K=F\cap \Q_p$ and $F/\Q$ is unramified at $p$. Note that $\Gal(F/\Q)=\Z/2\Z\times \Z/2\Z$ is non-cyclic, hence $2\mid [K:\Q]$, hence $\Gal(L_1/K)$ is abelian since $\Gal(L_1/\Q)=\D_4$. Therefore we can lift $\chi$ to be a character $\chi'$ of $\Gamma_K$ factoring through $\Gal(L_1/K)$ and we have
$$\begin{aligned}
\res_{\Q/K}(A)&=\sum_\sigma \text{cor}_{V/U}(N_{L_2/F}(z_2)^{\sigma},\res_{U/V}(\chi'))=\sum_\sigma (N_{F/K}(N_{L_2/F}(z_2)^{\sigma}),\chi')\\
&=(N_{F/\Q}(N_{L_2/F}(z_2)),\chi')=(N_{L_2/\Q}(z_2),\chi')\\
&=(97,\chi')-(N_{L_1/\Q}(z_1),\chi')=(97,\chi').
\end{aligned}
$$
Hence we have $A(x_p)=0$ for $p \neq 2, 97, \infty$, and $A(x_p) \neq 0$ for $p=97$,
since $\left( \frac{2}{97} \right)_4 = -1$.

Eventually, we deduce that $A$ is in the unramified Brauer group of $X$ and that $\sum_p A(x_p) = A(x_{97}) \neq 0$ for all $(x_p) \in \prod_p X(\Q_p)$. Therefore, the reciprocity law from class field theory implies that $X(\Q) = \emptyset$, which concludes the proof.
\end{example}

\bf{Acknowledgment} \it{The first author acknowledges the support of the French Agence Nationale de la Recherche (ANR) under reference ANR-12-BL01-0005. He also thanks Ulrich Derenthal and the "Center for Advanced Studies" for the invitation to Ludwig-Maximilians-Universit\"at in Munich where this work was initiated. The second author is supported by NSFC grant \# 10671104, 973 Program 2013CB834202 and grant DE 1646/2-1 of the Deutsche Forschungsgemeinschaft. Both authors thank Jean-Louis Colliot-Th\'el\`ene and the referee for useful remarks and suggestions.}



\bigskip
{\small

{\scshape
C. Demarche: Institut de Math\'ematiques de Jussieu (IMJ), Universit\'e Pierre et Marie Curie, 4 place Jussieu,
75252 Paris Cedex 05, France}
\smallskip

{\it E-mail: }
\url{demarche@math.jussieu.fr}

\bigskip

{\scshape
D. Wei: Academy of Mathematics and System Science,  CAS, Beijing
100190, P.R.China and Mathematisches Institut der Universit\"at M\"unchen Theresienstr. 39, D-80333 M\"unchen}
\smallskip

{\it E-mail: }
\url{dshwei@amss.ac.cn}
}

\end{document}